\documentclass[11pt]{article}
\usepackage[T1]{fontenc}
\usepackage[margin=1in]{geometry}
\usepackage[utf8]{inputenc}
\usepackage{lmodern,amsmath,amssymb,amsthm,mathtools}
\usepackage{microtype,enumitem,booktabs,tabularx}
\usepackage{xcolor}
\definecolor{MidnightBlue}{cmyk}{0.98,0.13,0,0.43}
\usepackage{tikz-cd}
\usepackage{bookmark}
\usepackage{hyperref}
\hypersetup{
    colorlinks=true,
    linkcolor=MidnightBlue,
    citecolor=MidnightBlue,
    urlcolor=MidnightBlue
}
\numberwithin{equation}{section}
\newtheorem{theorem}{Theorem}[section]

\newtheorem{proposition}{Proposition}[section]
\newtheorem{lemma}{Lemma}[section]

\newtheorem{corollary}[theorem]{Corollary}
\theoremstyle{definition}
\newtheorem{definition}[theorem]{Definition}
\theoremstyle{remark}

\newcommand{\D}{\mathbb D}
\newcommand{\T}{\mathbb T}
\newcommand{\C}{\mathbb C}
\newcommand{\ol}{\overline}
\renewcommand{\cal }{\mathcal}
\newcommand{\E}{\mathbb E}

\newcommand{\dd}{\,\mathrm d}
\newcommand{\ip}[2]{\langle #1,#2\rangle}
\newcommand{\norm}[1]{\lVert #1\rVert}
\newcommand{\abs}[1]{\lvert #1\rvert}

\DeclareMathOperator{\Rea}{Re}

\setlist[enumerate]{label=\textup{(\roman*)},leftmargin=*,itemsep=3pt}
\title{On the existence of independent inner functions}

\author{Xiangdi Fu, Kunyu Guo, Dilong Li and Yanqi Qiu}
\date{}

\hypersetup{
  pdftitle={On the existence of independent inner functions},
  pdfauthor={Xiangdi Fu, Kunyu Guo, Dilong Li, Yanqi Qiu}
}

\begin{document}
\maketitle

\begingroup
\renewcommand{\thefootnote}{}
\footnotetext{%
}
\endgroup

\begin{abstract}
The existence of independent inner functions is established: there exist inner functions $\varphi$ and $\psi$, both
vanishing at the origin, such that
$$
\frac{1}{2\pi}\int_0^{2\pi}
\cramped{\varphi^m(e^{i\theta})}
\overline{\psi^n(e^{i\theta})}\,\dd\theta = 0,
\qquad m,n \geq 1.
$$
In fact, both $\varphi$ and $\psi$ can be chosen to be Blaschke products. The proof proceeds by an iterative construction, which suggests a general approach to produce inner functions with a desired joint distribution.
\end{abstract}

\section{Introduction}

\smallskip
\noindent\textbf{Background.}
Let $\D\subset \C$ be the open unit disk, and let $\T=\partial \D$ be the unit circle, equipped with the normalized Lebesgue measure $d\sigma$. For $1\leq p \leq \infty$, let $H^p$ denote the classical Hardy space: $$H^p:= \left\{ f\in L^p(\T): \int_\T f(w) \overline{w}^{n} d\sigma(w) = 0 \text{ for } n <0\right\}.$$ As usual, any function $f\in H^p$ is identified with its analytic extension $$f(z):= \int_\T f(w) P_z(w) d\sigma(w),  \quad z\in \D,$$ where $P_z(w)= \frac{1-|z|^2}{|1-\overline{z}w|^2}$ denotes the Poisson kernel. 

A function $\varphi \in H^\infty$ is called \emph{inner} if $|\varphi(w)| = 1$ for almost every $w\in \T$. Two inner functions are called \emph{independent} if they are independent when viewed as random variables on the probability space $(\T,\sigma)$. We are interested in the following problem:

\begin{quote}
  \emph{Do there exist two (nonconstant) independent inner functions?}
\end{quote}

This question was posed by Forelli and was  studied in Richards's 1974 thesis \cite{Richards1974}. Richards showed that if there exist two independent inner functions, then there exists a countably infinite family of such functions \cite[Theorem 2.12]{Richards1974}, and linked Forelli's problem to isometric embeddings between Hardy spaces on tori of different dimensions. Later, Rudin pointed out a natural connection between Forelli's
question and the following problem.
\begin{quote}
\emph{Let $f$ be an $H^\infty$-function satisfying 
\begin{align}\label{eq:Rudin-orthogonality-function}
\int_\T \cramped{f^m}\,\overline{f^n}\,d\sigma = 0
\qquad \text{for all } m,n\geq 0 \text{ with } m\neq n.
\end{align}
Is $f$ necessarily a multiple of an inner function vanishing at the origin?}
\end{quote} 

In fact, any pair of nonconstant independent inner functions $(\varphi,\psi)$ can be normalized (by post-composition with a suitable M\"obius transformation) so that $\varphi(0)=\psi(0)=0$, and the independence is then equivalent to $$\int_\T \cramped{\varphi^m} \overline{\psi^n} d\sigma =0, \quad m,n \geq 1.$$ Now if such $\varphi$ and $\psi$ exist, then an easy calculation shows that
$
    f:=\varphi+\psi
$
satisfies \eqref{eq:Rudin-orthogonality-function}, but $f$ is not a constant multiple of an inner function, thereby providing a counterexample to Rudin's problem. We note that the above connection between these two problems was also discussed in an unpublished preprint of Cima, Korenblum and Stessin \cite{CimaKorenblumStessin1993}.

Although counterexamples (without relying on the existence of independent inner functions) to Rudin's orthogonality problem have been constructed by Sundberg \cite{Sundberg2003} and Bishop \cite{Bishop2005}, the existence of independent inner functions had remained unresolved.

\smallskip\noindent\textbf{Main results.} The main aim of this note is to show that such independent inner functions do exist, in a stronger form.

\begin{theorem}\label{thm:main}
There exist inner functions $\varphi,\psi$ such that
\begin{align}\label{eq:Blaschke-independent}
  \varphi(0)=\psi(0)=0,
  \qquad
  \int_{\T} \cramped{\varphi^m} \overline{\psi^n}\dd \sigma=0
  \quad m,n\geq 1.
\end{align}
Moreover, both $\varphi$ and $\psi$ can be chosen to be Blaschke products.
\end{theorem}

The proof of Theorem~\ref{thm:main} is existential. In fact, a pair of independent inner functions is obtained as the limit of a ``ping-pong'' iteration, whose two ``players'' are Propositions~\ref{prop:inner-approximation} and~\ref{prop:correction} below. The argument thus naturally divides into three steps: establishing the two propositions in turn and then alternating their applications to construct the desired limit. To make the main ideas of the proof more transparent, we choose to first construct independent inner functions vanishing at the origin, and then explain how to refine the argument to obtain Blaschke products with the
same properties.
\smallskip

It should be noted that any pair of independent inner functions vanishing at the origin must exhibit highly irregular boundary behavior. Indeed, independence implies that
$\{\varphi^m\psi^n:m,n\ge0\}$ is an orthonormal family in $H^2$.
Applying Bessel's inequality to the reproducing kernel for $H^2$ gives
\[
  \frac{1}{(1-|\varphi(z)|^2)(1-|\psi(z)|^2)}
  \leq \frac{1}{1-|z|^2},\quad z\in \D.
\]
Since $\psi$ is inner, for almost every $w\in\T$,
$$
  \frac{1-|\varphi(rw)|^2}{1-r^2}
  \ge\frac{1}{1-|\psi(rw)|^2}
  \longrightarrow\infty
  \qquad \text{as }r\to 1.
$$
By the Julia--Carath\'eodory theorem
\cite[Theorem~3.1]{Saksman2007}, the angular derivative of $\varphi$ is infinite almost everywhere on $\T$.
In particular, $\varphi$ admits no analytic continuation across any nonempty open arc of $\T$, and $\varphi'$ does not belong to the Nevanlinna class. The latter assertion is equivalent, by a theorem of Craizer \cite{Craizer1991}, to $\varphi$ having infinite Kolmogorov--Sinai entropy.

\smallskip\noindent\textbf{Further applications and related problems.}
As pointed out in \cite{Richards1974} and \cite{Wolniewicz1987}, if $(\varphi,\psi)$ is a pair of independent inner functions vanishing at the origin, then
$
\varphi_n := \varphi \circ \psi^{\circ n}
$ form a sequence of jointly independent inner functions, where $\psi^{\circ n}$ denotes the $n$-th iteration of $\psi$.
Thus the map
\begin{align*}
\Phi:\T\longrightarrow\T^\infty,\qquad
z\mapsto(\varphi_1(z),\varphi_2(z),\ldots)
\end{align*}
preserves the Haar measure. In particular, the induced composition operator gives an isometric embedding of the corresponding Hardy space on $\T^\infty$ into that on $\T$.

\begin{corollary}\label{cor:corollaries-of-inner}
For any $0<p\leq \infty$, the map 
\begin{align*}
  J_\Phi: H^p(\T^\infty) \longrightarrow H^p(\T); \quad F\mapsto F\circ \Phi, 
\end{align*} 
is an isometric embedding. In particular, composing $J_\Phi$ with the natural embedding $H^p(\T^d)\hookrightarrow H^p(\T^\infty)$ yields an
isometric embedding of $H^p(\T^d)$ into $H^p(\T)$ for every finite $d$. 
\end{corollary}
\noindent This interesting corollary suggests a possible route for reducing questions about Hardy spaces on higher-dimensional tori to the one-variable setting; we refer to \cite{Defant2019,Queffelec2013} for the general theory of Hardy spaces on the infinite-dimensional torus and its natural connections to Diophantine approximation and Dirichlet series.

We also believe that the existence of independent inner functions
may shed light on other problems. Indeed, building on Theorem~\ref{thm:main}, the first and third authors, together with Kenan Zhang \cite{FuLiZhang2026}, have obtained (with assistance from ChatGPT Astra) a complete characterization of the measures on $\C$ arising as distributions of functions in $H^\infty(\Omega)$, where $\Omega$ is either $\T^d$ or the unit sphere $\partial\mathbb B_d$ in $\C^d$, thereby resolving a problem posed by Sundberg in \cite{Sundberg2003}.

A Blaschke product is determined, up to a unimodular constant, by its
zeros, counted with multiplicity. Nevertheless, it seems difficult to
write down an explicit pair of Blaschke products satisfying \eqref{eq:Blaschke-independent}. We are curious to understand
\begin{quote}
  \emph{What geometric conditions on the zero sequences, or what underlying principles, determine the independence of the corresponding Blaschke products?}
\end{quote}
Motivated by Marshall's celebrated theorem that the closed unit ball of $H^\infty$ is the norm-closed convex hull of the Blaschke
products \cite{Marshall1976}, we also pose the following problem.

\begin{quote}
\emph{Consider the sets of all pairs of independent inner functions (Blaschke products) vanishing at the origin. How large are these sets in $zH^\infty\times zH^\infty$?}
\end{quote}

\smallskip
\noindent\textbf{AI Disclosure.} This project began in August 2025. Before AI tools came into widespread use, we had observed several necessary conditions for inner functions to be independent, including those mentioned following
Theorem~\ref{thm:main} (which should be familiar to experts) and others involving Aleksandrov--Clark measures. We had also explored several approaches aimed at proving the nonexistence of independent inner functions, but without success.

From June 2026, we began exploring this problem extensively with OpenAI's ChatGPT-5.5, Sol, and Astra, initially with limited substantive progress. In mid-September, while examining a proof provided by Astra from our work on another problem concerning Hardy spaces on the infinite-dimensional torus, we identified a useful clue for the present problem. Further discussions with Astra then led to an initial proof of Theorem~\ref{thm:main}.

The proofs provided by Astra were mathematically complete. We only simplified and reorganized some of the arguments to improve the readability, and identified some connections not explicitly emphasized by ChatGPT including the Poisson point process underlying \(\widetilde{\Pi}\), Carath\'eodory--Schur interpolation in Proposition \ref{prop:inner-approximation}, and the Moore--Penrose pseudoinverse in Proposition \ref{prop:correction}.

The authors have independently verified and understood all arguments. The final manuscript was written and organized entirely by the authors, who take full responsibility for its contents.

\section{Independent inner functions}
\label{sec:independent-inner-functions}

Given a pair of inner functions $(\varphi, \psi)$, both vanishing at the origin, let $\mu_{\varphi,\psi}$ be the pushforward measure on $\T^2$ determined by $$\mu_{\varphi,\psi}(F) := \int_{\T} F(\varphi(w),\psi(w)) d\sigma(w), \quad F \in C(\T^2).$$ Since $\varphi$ and $\psi$ are analytic, the Fourier coefficients of $\mu_{\varphi,\psi}$ are given by 
\begin{align}\label{eq:Fourier-joint-distribution}
  \int_{\T^2} \zeta^m \overline \eta^{n} d\mu_{\varphi,\psi} (\zeta, \eta)=
  \begin{cases}
    1, & m=n=0,\\
    \langle \varphi^m , \psi^n \rangle, & m, n \geq 1 \text{ or } m,n \leq -1\\
    0, & \text{otherwise}.
  \end{cases}
\end{align}
We denote by $\cal W$ the Wiener algebra on $\T^2$ consisting of all complex Borel measures whose Fourier coefficients are absolutely summable. Clearly, each $\mu \in \cal W$ is automatically absolutely continuous. Moreover, we shall always identify integrable functions with absolutely continuous Radon measures. 

Let $P_z$ be the Poisson kernel at $z\in \D$, and let $f,g$ be $H^\infty$-functions satisfying 
\begin{align}\label{eq:schur-function}
  f(0)=g(0)=0, \quad \|f\|_\infty < 1, \quad \|g\|_\infty < 1.
\end{align} 
The related \emph{Poisson balayage} $\rho_{f,g}$ is defined by
\begin{equation*}
  \rho_{f,g}(\zeta,\eta)
  :=\int_{\T}P_{f(z)}(\zeta)P_{g(z)}(\eta)d \sigma (z),\quad (\zeta, \eta) \in \T^2.
\end{equation*}
By Fubini's theorem and the Poisson integral formula, the Fourier coefficients of $\rho_{f,g}$ are given by
\begin{align}\label{eq:Fourier-coefficient-balayage}
  \int_{\T^2}\cramped{\zeta^m}\overline{\eta^n}\rho_{f,g} \dd \sigma_2
  =\begin{cases}
    1, & m=n=0;\\
    \ip{f^m}{g^n}, & m,n \geq 1;\\
    \overline{\ip{f^{-m}}{g^{-n}}}, & m,n \leq -1;\\
    0, & \text{otherwise}.
  \end{cases}
\end{align}
Note that $\rho_{f,g} \in \cal W$, since we have assumed both $\|f\|_\infty$ and $\|g\|_\infty$ are strictly less than one.

\subsection{An approximation lemma for joint distributions}

To prove the first part of Theorem \ref{thm:main}, it is equivalent to show there exist inner functions $\varphi$, $\psi$, vanishing at the origin, such that $\mu_{\varphi,\psi} = \sigma_2$ (the normalized Lebesgue measure on $\T^2$). Our first step is to show that the pushforward measure of a pair of inner functions can approximate any fixed Poisson balayage $\rho_{f,g}$ for $f$ and $g$ satisfying \eqref{eq:schur-function}.

\begin{proposition}
\label{prop:inner-approximation}
Let $f,g$ be $H^\infty$-functions satisfying \eqref{eq:schur-function}. Then for any $\varepsilon>0$ there
exist inner functions $\varphi,\psi$ with $\varphi(0)= \psi(0) = 0$, such that
\begin{align}\label{eq:inner-approximation}
  \mu_{\varphi,\psi}\in \cal W, \quad \norm{\mu_{\varphi,\psi}-\rho_{f,g}}_\cal W<\varepsilon. 
\end{align}
Moreover, $\varphi $ and $\psi$ can be chosen so that 
\begin{align}\label{eq:inner-approximation-norm-estimate}
  \norm{\varphi-f}_2^2 < 10 (1-\norm f_2^2),
  \qquad
  \norm{\psi-g}_2^2 < 10 (1-\norm g_2^2).
\end{align}
\end{proposition}

The proof of Proposition \ref{prop:inner-approximation} requires a random construction of inner functions using the \emph{Poisson point process}. All the results on Poisson point processes used in this note are classical and can be found in \cite{LastPenrose}. We do not review the general theory here, but have chosen a presentation intended to be accessible to readers who are unfamiliar with Poisson point processes.

Let $\Pi $ be a Poisson random measure on $\T\times(0,\infty)$
with intensity
\[
    d\nu(\zeta,a)
    =\lambda\,\dd \sigma(\zeta)\frac{da}{a^{1+\beta}},
    \qquad \lambda>0, \quad 0<\beta<1.
\] 
In particular, the following equivalence will be useful: for any non-negative measurable function $H$ on $\T\times (0,\infty)$, it holds that (see, e.g., \cite[Proposition 12.1]{LastPenrose})
\begin{equation}\label{eq:poisson-integrability}
  \int H\dd \Pi<\infty\,\,\text{almost surely}
  \quad \Longleftrightarrow \quad \int(1\wedge H)\dd\nu<\infty.
\end{equation}

The equivalence \eqref{eq:poisson-integrability} implies that, almost surely, the positive linear functional $$h\mapsto \int_{\T \times (0,\infty)} a h(\zeta) \dd \Pi (\zeta,a)$$ is bounded on $C(\T)$ and thus  induces a random Borel measure $\widetilde{\Pi}$ on $\T$ such that $$\int_\T h \dd\widetilde{\Pi} =\int_{\T \times (0,\infty)} a h(\zeta) \dd \Pi (\zeta,a).$$ Note that $\widetilde{\Pi}$ is also a positive random measure concentrated on a countable set. 

Next, we consider the singular inner function associated with
$\widetilde{\Pi}$, which we call a \emph{Poisson random inner function}. This random inner function has some interesting function-theoretic properties (see Lemma \ref{lem:oscillation} below), and we believe that it is worthy of further study.

\begin{definition}
Let $\widetilde{\Pi}$ be the random measure on $\T$ defined above. The associated random singular inner function $$S_{\widetilde{\Pi}}(z) := \exp\Big( - \int_{\T} \frac{\zeta+z}{\zeta-z} \dd \widetilde{\Pi}\Big), \quad z\in \D,$$ is called a Poisson random inner function.
\end{definition} 

The proof of Proposition \ref{prop:inner-approximation} requires the following properties of $S_{\widetilde{\Pi}}$.
\begin{lemma}\label{lem:oscillation}
Let $\beta\in (0,1)$ and $\lambda>0$. Let $S_1$ and $S_2$ be two independent copies of $S_{\widetilde{\Pi}}$. Then there exists a constant $C>0$, depending only on $\beta$, such that 
\begin{equation}\label{eq:oscillation}
  \E\bigg[ \bigg|\int_{\T}b \, \cramped{S_1^m} \overline{S_2^n}\dd \sigma \bigg|^2 \bigg]
  \leq C \|b\|_\infty^2  \, \big[ \lambda (m^\beta + n^\beta)\big]^{-\frac{1}{1-\beta}}
\end{equation}
for every (deterministic) $b\in L^\infty(\T)$ and all integers $m,n \geq 0$ with $(m,n)\neq (0,0)$.
\end{lemma} 

Let us begin by proving Proposition \ref{prop:inner-approximation} on the assumption that Lemma \ref{lem:oscillation} holds, and then proceed to prove Lemma \ref{lem:oscillation} in Section \ref{sect:random-inner-function}.

\begin{proof}[Proof of Proposition \ref{prop:inner-approximation}]

We first prove the case where both $f$ and $g$ are polynomials of degree less than $N$ and satisfy \eqref{eq:schur-function}. To obtain the desired inner functions whose joint distribution is close to $\rho_{f,g}$, the idea is to construct inner functions $\varphi$ and $\psi$ such that 
\begin{itemize} 
  \item their first $N$ Taylor coefficients agree with those of $f$ and $g$ respectively; 
  \item the contribution of the remaining Taylor coefficients to the error $\big| \ip{\varphi^m}{\psi^n} - \ip{f^m}{g^n}\big|$ is as small as possible.
\end{itemize}
The first point can be achieved by following Carath\'eodory--Schur interpolation \cite{Garnett2007}, while the second requires an application of Lemma \ref{lem:oscillation}.

Now for polynomials $f$ and $g$ of degree less than $N$, we consider
$$
  \tau_f(z,w):=\frac{f(z)-z^{2N}w}{1-\overline{f(1/\overline{z})}z^{2N}w},\qquad \tau_g(z,w):=\frac{g(z)-z^{2N}w}{1-\overline{g(1/\ol z)}z^{2N}w}.
$$
Both $\tau_f$ and $\tau_g$ are rational inner functions in two variables \cite{Rudin1969}, and their powers admit the following series expansions:
\begin{align}\label{eq:series-expansion-tau}
  \tau_f^m(z,w)=\sum_{k\geq0}a_{m,k}(z)w^k, \qquad \tau_g^m(z,w)=\sum_{k\geq0}b_{m,k}(z)w^k,
\end{align}
which converge absolutely on $\overline{\D}\times\overline{\D}$. We claim that there exists a positive constant $C$, depending only on $\|f\|_\infty$ and $\|g\|_\infty$, such that 
\begin{align}\label{eq:consequence-Cauchy-estimate}
  \sum_{m\geq 1} \|a_{m,k}\|_\infty \leq C (k+1), \quad \sum_{m\geq 1} \|b_{m,k}\|_\infty \leq C (k+1).
\end{align}

The verification of \eqref{eq:consequence-Cauchy-estimate}
only requires elementary Cauchy estimates. Since $a_{m,0} = f^m$, $b_{m,0} = g^m$, we only need to consider the case $k\geq 1$. For any fixed $z\in\T$, applying Cauchy's coefficient formula for $\tau_f(z,\cdot)^m$ gives $$
  a_{m,k}(z)
  =\frac{1}{2\pi i}\int_{|w|=r}
    \frac{\tau_f(z,w)^m}{w^{k+1}}\,dw.
$$
Since for each fixed $z$, $\tau_f(z,\cdot)$ is an automorphism of $\ol{\D}$, one observes that
$$
  R:= \sup_{z\in \T}\sup_{|w|=1/2}|\tau_f(z,w)|
  \leq \frac{\|f\|_\infty+1/2}{1+\|f\|_\infty/2}<1.
$$
By setting $r=1$ and $r=1/2$ in Cauchy's coefficient formula, we obtain 
$$
  \|a_{m,k}\|_\infty \leq \min\{ 1, 2^k R^m\}
$$
and hence 
\[
\begin{aligned}
  \sum_{m\geq1}\|a_{m,k}\|_\infty \leq  \sum_{2^k R^m \geq 1} 1 + \sum_{2^k R^m <1} 2^k R^m \leq C(R)k.
\end{aligned}
\]
Applying the same argument to $g$ proves the claim \eqref{eq:consequence-Cauchy-estimate}.

Now let $S_1$ and $S_2$ be two independent copies of the Poisson random inner function as stated in Lemma \ref{lem:oscillation}, and put 
\begin{align}\label{eq:random-inner-approximation}
  \varphi(z):=\tau_f(z,S_1(z)),
  \qquad
  \psi(z):=\tau_g(z,S_2(z)).
\end{align}
It is clear that $\varphi$ and $\psi$ are random inner functions, $\varphi(0)=\psi(0)=0$, and 
$$
  \varphi-f
  =
  z^N S_1\frac{ff^\sharp-z^N}{1-f^\sharp z^N S_1}
  \in z^{N} H^2; 
$$ here $f^\sharp(z)=\ol{f(1/\ol z)} z^N$, which is an analytic polynomial. Now, since $\operatorname{deg} f<N$, one has $\varphi-f \perp f$. It also holds that $\psi-g\perp g$. Therefore,
\begin{align}\label{eq:polynomial-orthogonality}
  \|\varphi - f\|^2_2 = 1-\|f\|^2_2 \text{\, and \,}
  \|\psi-g\|_2^2=1-\|g\|_2^2 \,\,\text{ almost surely}.
\end{align}
Noting that $a_{m,0}=f^m$, $b_{n,0}=g^n$ in \eqref{eq:series-expansion-tau}, we see $$\Delta_{m,n} := \ip{\varphi^m}{\psi^n} - \ip{f^m}{g^n} = \sum_{\substack{k,\ell \geq 0 \\ (k,\ell)\neq (0,0)}} \int_{\T}
  a_{m,k}\overline{b_{n,\ell}}\,
  S_1^k\overline{S_2}^{\,\ell}\,\dd\sigma.$$
  By Lemma \ref{lem:oscillation}, 
  \begin{align*} 
    \E |\Delta_{m,n}| \leq & \sum_{\substack{k,\ell \geq 0 \\ (k,\ell)\neq (0,0)}} \E \bigg| \int_\T a_{m,k} \ol{b_{n,\ell}} S_1^k \ol{S_2}^\ell \dd \sigma \bigg| \\ \leq & \sum_{\substack{k,\ell \geq 0 \\ (k,\ell)\neq (0,0)}} \bigg[\E \bigg|  \int_\T a_{m,k} \ol{b_{n,\ell}} S_1^k \ol{S_2}^\ell \dd \sigma \bigg|^2 \bigg]^{1/2} \\ 
    \lesssim & \sum_{\substack{k,\ell \geq 0 \\ (k,\ell)\neq (0,0)}} \|a_{m,k}\|_\infty \|b_{n,\ell}\|_\infty \big[ \lambda (k^\beta + \ell^{\beta})\big]^{-\frac{1}{2(1-\beta)}}
  \end{align*} 
  Summing over $m,n\geq 1$, and then using \eqref{eq:Fourier-joint-distribution}, \eqref{eq:Fourier-coefficient-balayage} and \eqref{eq:consequence-Cauchy-estimate}, we obtain 
\begin{align*}
  \E\|\mu_{\varphi,\psi} - \rho_{f,g}\|_\cal W = & 2 \sum_{m,n\geq 1 }  \E |\Delta_{m,n}| 
  \lesssim  \sum_{m,n\geq 1 } \sum_{\substack{k,\ell \geq 0 \\ (k,\ell)\neq (0,0)}} \|a_{m,k}\|_\infty \|b_{n,\ell}\|_\infty \big[ \lambda (k^\beta + \ell^{\beta})\big]^{-\frac{1}{2(1-\beta)}}\\
  = & \sum_{\substack{k,\ell \geq 0 \\ (k,\ell)\neq (0,0)}} \big[ \lambda (k^\beta + \ell^{\beta})\big]^{-\frac{1}{2(1-\beta)}}  \bigg( \sum_{m\geq 1}  \|a_{m,k}\|_\infty \bigg) \bigg( \sum_{n\geq 1}\|b_{n,\ell}\|_\infty \bigg)\\
  \lesssim & \sum_{\substack{k,\ell \geq 0 \\ (k,\ell)\neq (0,0)}} (k+1) (\ell+1) \big[ \lambda (k^\beta + \ell^{\beta})\big]^{-\frac{1}{2(1-\beta)}}.
\end{align*}
As long as $\beta$ is sufficiently close to $1$ (one can henceforth fix $\beta = 15/16$), then $$\sum_{\substack{k,\ell \geq 0 \\ (k,\ell)\neq (0,0)}} (k+1)( \ell+1) \big[  (k^\beta + \ell^{\beta})\big]^{-\frac{1}{2(1-\beta)}} < \infty.$$ In particular, for sufficiently large $\lambda$, $$\E\|\mu_{\varphi,\psi} - \rho_{f,g}\|_\cal W \leq C(f,g,\beta) \lambda^{-\frac{1}{2(1-\beta)}} < \varepsilon.$$  Thus there exist two inner functions satisfying \eqref{eq:inner-approximation}, in the case where $f$ and $g$ are fixed polynomials.

Now for the general $f$ and $g$ satisfying \eqref{eq:schur-function}, we consider their Ces\`aro means $f_N$ and $g_N$. It is routine to verify $\rho_{f_N, g_N} \to \rho_{f,g}$ in $\cal W$ as $N\to \infty$. Thus  one can fix a sufficiently large $N$ so that $\|\rho_{f_N,g_N} - \rho_{f,g}\|_\cal W <\varepsilon/2$ and $$\|f_N -f \|_2< \sqrt{1- \|f\|^2_2} \quad 1-\|f_N\|^2_2< 2 (1-\|f\|_2^2).$$ Let $\varphi$, $\psi$ be two inner functions constructed as above for polynomials $f_N$ and $g_N$. Then $\|\mu_{\varphi,\psi} - \rho_{f, g}\|_\cal W < \varepsilon$ and (by \eqref{eq:polynomial-orthogonality}) 
\begin{align*} 
  \|\varphi - f\|_2 & \leq \|\varphi-f_N\|_2 + \|f_N - f\|_2\\
  & =\sqrt{1-\|f_N\|_2^2}+\|f_N - f\|_2< 3 \sqrt{1-\|f\|^2_2}.
\end{align*} Similarly, $\|\psi - g\|_2 < 3 \sqrt{1-\|g\|^2_2}$. The proof of Proposition \ref{prop:inner-approximation} is completed.
\end{proof}

\subsection{A superlinear perturbation}
\label{sec:correction}

Recall that Proposition \ref{prop:inner-approximation} guarantees that for any given $f,g$ satisfying \eqref{eq:schur-function}, the Poisson balayage $\rho_{f,g}$ can be approximated by $\mu_{\varphi,\psi}$ in $\cal W$ for some pair of inner functions. The next key proposition shows that, for each $M=M(\varepsilon)$ as specified below, the pair $(\varphi,\psi)$ admits a perturbation of size $O(M^2\varepsilon)$ in \(H^2\) that reduces the error in the first $M\times M$ mixed moments from linear to
superlinear order. The precise statement is as follows.

\begin{proposition}
\label{prop:correction}
Let $0<\varepsilon<e^{-2026}$ and let $M:=M(\varepsilon)=\lfloor \varepsilon^{-1/2026}\rfloor$. Then for any given inner functions $\varphi,\psi$ satisfying $$\varphi(0)=\psi(0) = 0, \quad \norm{\mu_{\varphi,\psi}-1}_\cal W < \varepsilon,$$ there exist small perturbations $h,k\in zH^\infty$ such that
\begin{align}\label{eq:perturbation-bounds}
  \norm h_2+\norm k_2\leq CM^2\varepsilon, \quad \norm h_\infty,\norm k_\infty\leq\sqrt\varepsilon,
\end{align}
and
\begin{equation}\label{eq:finite-correction}
  \sum_{1\leq m,n\leq M}
  \abs{\ip{(\varphi+h)^m}{(\psi+k)^n}}
  \leq CM^{7}\varepsilon^{3/2}.
\end{equation}
Here the constant $C$ is universal.
\end{proposition} 

The proof of Proposition \ref{prop:correction} will be carried out in the framework of the \emph{Moore--Penrose inverse}, combined with elementary function-theoretic estimates in Hardy spaces. To make the underlying idea more transparent, we divide the proof of
Proposition~\ref{prop:correction} into three steps: 
\begin{itemize}
    \item For small bounded analytic
    perturbations $h,k\in zH^\infty$, we write
    \begin{align}\label{eq:perturbation-equation}
        \ip{(\varphi+h)^m}{(\psi+k)^n}
        = \ip{\varphi^m}{\psi^n}+ 
        \underbrace{m\ip{\varphi^{m-1}h}{\psi^n}+n\ip{\varphi^m}{\psi^{n-1}k}}_{\text{denoted by }T_{m,n}(h,\ol{k})}
        +R_{m,n}(h,\ol{k}).
    \end{align}
    The remainders $R_{m,n}(h,\ol{k})$ can be well-controlled by Taylor's theorem with integral remainder. Thus we only need to seek a good perturbation to make the sum $\ip{\varphi^m}{\psi^n} + T_{m,n}(h,\ol{k})$ as small as possible.

    \item
    We then consider the linear operator
    \begin{align}\label{eq:definition-T}
        T:zH^2\oplus\overline{zH^2}
        \longrightarrow \operatorname{Mat}_M(\C); \quad (h,\ol{k})\mapsto \big[T_{m,n}(h,\ol{k}) \big]_{M\times M},
    \end{align} 
    and apply the Moore--Penrose inverse to obtain the minimal $L^2$-norm solution $(h_0,\ol{k_0})$ of the equation $T(h,\ol{k})= - Y,$ where $Y := [\ip{\varphi^m}{\psi^n}]_{M\times M}$.

    \item The good news is that $Y+T(h_0,\ol{k_0})$ vanishes, but the bad news is $h_0$ or $k_0$ may not be bounded. To obtain the final $H^\infty$-perturbations, we will apply an analytic truncation trick using the Hilbert transform.
\end{itemize}

We now turn to the details.

\begin{proof}[Proof of Proposition \ref{prop:correction}] Let $T$ be the linear operator defined in \eqref{eq:perturbation-equation} and \eqref{eq:definition-T}, and denote by $\|\cdot \|_{HS}$ the Hilbert-Schmidt norm on $\operatorname{Mat}_M(\C)$.

\medskip
\noindent\textit{Step 1.} For $0\leq t\leq1$, an elementary calculation shows 
\begin{align*}
  \frac{\dd^2}{\dd t^2}
      \ip{(\varphi+t h)^m}{(\psi+t k)^n}
  =&m(m-1) \ip{h^2 (\varphi+th)^{m-2}}{(\psi+tk)^n}\\
  & + 2mn \ip{h(\varphi+th)^{m-1}}{k (\psi+tk)^{n-1}}\\
  & + n(n-1) \ip{(\varphi+t h)^{m}}{k^2 (\psi+tk)^{n-2}}.
\end{align*}
Substituting the trivial bounds $m,n \leq M$ and $$\|\varphi + th\|_\infty \leq 1+\|h\|_\infty,\quad \|\psi + tk\|_\infty \leq 1+\|k\|_\infty,$$ and then combining with the mean-value theorem, one can deduce that $$|R_{m,n}(h,\ol{k})| \lesssim M^2 (1+\|h\|_\infty)^M (1+\|k\|_\infty)^M (\|h\|_2^2 + \|k\|_2^2)$$ Note that $(1+\sqrt{\varepsilon})^M \leq e^2$. Thus for those perturbations satisfying \eqref{eq:perturbation-bounds}, it always holds that 
\begin{align}\label{eq:remainder-estimate}
  \sum_{1\leq m,n \leq M} |R_{m,n}(h,\ol{k})| \lesssim M^8 \varepsilon^2.
\end{align}

\medskip
\noindent\textit{Step 2.} To apply the Moore-Penrose method, we need to bound the adjoint operator $T^*$ from below. Standard calculation shows the adjoint is of the form $$T^* = (P A , \ol{P} B ): \operatorname{Mat}_M (\C) \to zH^2 \oplus \ol{zH^2},$$ where $P$ and $\ol{P}$ denote the orthogonal projections from $L^2$ onto $zH^2$ and $\ol{zH^2}$ respectively, and the operators $A$, $B$:  $\operatorname{Mat}_M(\C)\to L^2$ are given by $$A: [c_{m,n}]_{M\times M} \mapsto \sum_{1\leq m,n\leq M} m c_{m,n} \ol{\varphi}^{m-1} \psi^n, \quad B: [c_{m,n}]_{M\times M} \mapsto \sum_{1\leq m,n\leq M} n c_{m,n} \ol{\varphi}^{m} \psi^{n-1}.$$ Let us compute the inner product of $Ac$ and $\varphi \psi Bc = \sum nc_{m,n} \ol{\varphi}^{m-1} \psi^n$. Indeed, one has
\begin{align*} 
    \ip{Ac}{\varphi\psi Bc} & = \sum_{1\leq m,n \leq M} \sum_{1\leq m',n' \leq M} mn' c_{m,n} \ol{c_{m',n'}} \ip{\ol{\varphi}^{m-1} \psi^n}{\ol{\varphi}^{m'-1} \psi^{n'}}\\
    & = \sum_{1\leq m,n \leq M} \sum_{1\leq m',n' \leq M} mn' c_{m,n} \ol{c_{m',n'}} \ip{\psi^{n-n'}}{\varphi^{m-m'}}.
\end{align*}
Since $\varphi(0)=\psi(0)=0$ and $\|\mu_{\varphi,\psi}-1 \|_\cal W <\varepsilon$, we have $$\ip{\psi^{n-n'}}{\varphi^{m-m'}}= \begin{cases}
    1, & \text{if} \,\,m=m', n= n';\\
    0, & \text{if} \,\,m=m', n\neq n';\\
    0, & \text{if} \,\,m\neq m', n= n'.\\
\end{cases} \quad \text{and } \quad |\ip{\psi^{n-n'}}{\varphi^{m-m'}}|< \varepsilon, \,\, \text{if } m\neq m', n\neq n'.$$ By the triangle inequality, this implies

\begin{align*} \big|\ip{Ac}{\varphi\psi Bc}\big| & \geq \sum_{1\leq m,n \leq M} |c_{m,n}|^2 - M^2 \varepsilon \cdot \sum_{1\leq m,n \leq M} \sum_{1\leq m',n' \leq M}  |c_{m,n}| |c_{m',n'}|\\
& = \|c\|^2_{HS} - M^2\varepsilon \bigg( \sum_{1\leq m,n\leq M} |c_{m,n}|\bigg)^2 \geq (1-M^4\varepsilon) \|c\|^2_{HS}.
\end{align*}
Since $M^4\varepsilon<e^{-2022}$, we see 
\begin{align}\label{eq:inner-product-lower-bound}
  \big|\ip{Ac}{\varphi\psi Bc}\big| \gtrsim \|c\|_{HS}^2.
\end{align}
On the other hand, the inner product can be decomposed as 
\begin{align*}
  \ip{Ac}{\varphi \psi Bc} = \ip{PAc}{\varphi \psi Bc}+ \ip{(I-P)Ac}{\varphi \psi \ol{P}Bc} + \ip{(I-P)Ac}{\varphi\psi (I-\ol{P})Bc}.
\end{align*}
Observe that the last term vanishes, since $\varphi\psi (I-\ol{P})Bc \in z H^2$ and $(I-P)Ac \in \ol{H^2}$. Therefore, by the Cauchy-Schwarz inequality, \begin{align*}
  \big| \ip{Ac}{\varphi \psi Bc}\big| & \leq \|PAc\|_2 \|\varphi\psi Bc\|_2 + \|(I-P)Ac\|_2\|\varphi\psi \ol{P}Bc\|_2 \\
  & = \|PAc\|_2 \|Bc\|_2 + \|(I-P)Ac\|_2 \|\ol{P}Bc\|_2\\
  & \leq \|PAc\|_2 \|Bc\|_2 + \|\ol{P}Bc\|_2 \|Ac\|_2\\
  & \leq \Big( \|PAc\|^2_2 + \|\ol{P}Bc\|^2_2\Big)^{1/2} \Big( \|Ac\|^2_2 + \|Bc\|^2_2\Big)^{1/2}.
\end{align*}
Note that $$\|PAc\|^2_2 + \|\ol{P}Bc\|^2_2 = \|T^*c\|_2^2,\quad \text{and } \|Ac\|_2, \|Bc\|_2 \leq M^2 \|c\|_{HS},$$ and hence $$\big| \ip{Ac}{\varphi\psi Bc}\big| \lesssim M^2 \|T^*c\|_2\|c\|_{HS}.$$ Combining this with \eqref{eq:inner-product-lower-bound}, we conclude that, up to a universal constant, $T^*$ is bounded from below by $1/M^2$, and hence (by the Lax--Milgram theorem) $TT^*$ is invertible and $$\|(T T^*)^{-1}\| \lesssim M^4.$$ 

Recall that $Y=[\ip{\varphi^m}{\psi^n}]_{M\times M} \in \operatorname{Mat}_M(\C)$. Now if we set $(h_0,\ol{k_0}) := T^*(TT^*)^{-1}(-Y)$, then $(h_0,\ol{k_0})$ is the minimal $L^2$-norm solution of the equation $T(h,\ol{k}) + Y =0$ and its norm can be bounded by 
\begin{align}\label{eq:norm-estimate-minimal-solution}
\|h_0\|_2^2 + \|k_0\|_2^2  \lesssim M^4 \varepsilon^2. 
\end{align} Indeed, \eqref{eq:norm-estimate-minimal-solution} follows from
\begin{align*}
    \|(h_0,k_0)\|^2_2 &= \ip{T^*(TT^*)^{-1}Y}{T^*(TT^*)^{-1}Y} = \ip{(TT^*)^{-1}Y}{TT^*(TT^*)^{-1}Y}\\ 
    & = \ip{(TT^*)^{-1}Y}{Y} \lesssim M^4 \|Y\|_{HS}^2 \leq M^4 \varepsilon^2.
\end{align*}

\medskip
\noindent\textit{Step 3.} Next, we adopt a purely function-theoretic truncation procedure. Consider
$$
  H(z):=\int_{\T}\frac{\zeta+z}{\zeta-z}\abs{h_0(\zeta)}\dd \sigma(\zeta),
  \qquad \widetilde{h}(z):=\frac{h_0(z)}{1+H(z)/\sqrt{\varepsilon}}, \quad z\in \D.
$$
The function $\widetilde{k}$ is defined in the same manner. We claim that the functions $\widetilde{h}$ and $\widetilde{k}$ satisfy the required conditions in \eqref{eq:perturbation-bounds} and \eqref{eq:finite-correction}. 

It is clear that $\widetilde{h}$ is analytic and vanishes at the origin. Moreover, since $\operatorname{Re}H$ is the Poisson integral of $|h_0|$ and the subharmonicity ensures $|h_0(z)|\leq\operatorname{Re} H(z)$, we deduce that  
$$
   |\widetilde{h}(z)|\leq\frac{|h_0(z)|}{1+\Rea H(z)/\sqrt{\varepsilon}} \leq \frac{|h_0(z)|}{1+|h_0(z)|/\sqrt{\varepsilon}}\leq \min\Big\{\sqrt{\varepsilon}, |h_0(z)| \Big\}.
$$
The same estimate holds for $\widetilde{k}$. Thus the inequalities in \eqref{eq:perturbation-bounds} follow from \eqref{eq:norm-estimate-minimal-solution}. It only remains to verify \eqref{eq:finite-correction}. By the Cauchy-Schwarz inequality and the boundedness of the Hilbert transform on $L^2$, we have $$\|h_0-\widetilde{h}\|_1 = \bigg\|\frac{h_0 H/\sqrt{\varepsilon}}{1+H/\sqrt{\varepsilon}}\bigg\|_1 \leq \|h_0H/\sqrt{\varepsilon}\|_1 \leq \varepsilon^{-1/2}\|h_0\|_2\|H\|_2 \lesssim \varepsilon^{-1/2} \|h_0\|_2^2.$$ Combining this with \eqref{eq:norm-estimate-minimal-solution}, one has $\|h_0- \widetilde{h}\|_1 \lesssim M^4 \varepsilon^{3/2}$. The same bound also holds for $\|k_0-\widetilde{k}\|_1$. Thus the first two terms in \eqref{eq:perturbation-equation} can be bounded by 
\begin{align*} 
  & \bigg| \ip{\varphi^m}{\psi^n} + m\ip{\varphi^{m-1}\widetilde{h}}{\psi^n}+n\ip{\varphi^m}{\psi^{n-1}\widetilde{k}}\bigg| \\
  \leq & \bigg| m\ip{\varphi^{m-1}\widetilde{h}}{\psi^n}- m\ip{\varphi^{m-1}h_0}{\psi^n}\bigg| + \bigg| n\ip{\varphi^m}{\psi^{n-1}\widetilde{k}}-n\ip{\varphi^m}{\psi^{n-1}k_0}\bigg| \\
  \leq & M \|\widetilde{h}-h_0\|_1 + M \|\widetilde{k}-k_0\|_1 \lesssim  M^5\varepsilon^{3/2}.
\end{align*}
Substituting this estimate and \eqref{eq:remainder-estimate} into \eqref{eq:perturbation-equation}, we conclude that
\begin{align*} 
  &\sum_{1\leq m,n\leq M} \bigg|\ip{(\varphi+\widetilde{h})^m}{(\psi +\widetilde{k})^n}\bigg|  \lesssim M^7 \varepsilon^{3/2} + M^8\varepsilon^2 \lesssim M^7 \varepsilon^{3/2}.
\end{align*}
This completes the proof of Proposition \ref{prop:correction}.
\end{proof}

\subsection{The final iteration argument}
\label{sec:iteration}

By virtue of Proposition \ref{prop:inner-approximation} and Proposition \ref{prop:correction}, we are allowed to perform the following two operations:
\begin{itemize}
    \item For any given $f$ and $g$ in $H^\infty$ satisfying $$f(0)=g(0)=0,\quad \|f\|_\infty<1, \quad \|g\|_\infty <1,$$ we can find a pair of inner functions $(\varphi,\psi)$ whose joint distribution approximates $\rho_{f,g}$ in $\cal W$.
    \item For any pair of inner functions $(\varphi,\psi)$ whose joint distribution satisfies $\|\mu_{\varphi,\psi}-1\|_\cal W < \varepsilon$, there is a $\sqrt{\varepsilon}$-perturbation $(\varphi+h, \psi +k)$ so that the sum of the first $M\times M$-mixed moments of them can be controlled by $C M^7 \varepsilon^{3/2}$. In particular, if we rescale the perturbed functions by setting
    \[
        f:=\frac{t}{1+\sqrt{\varepsilon}}(\varphi+h),\qquad g:=\frac{t}{1+\sqrt{\varepsilon}}(\psi+k),
    \]
    for a suitable $t<1$, then 
    $\|f\|_\infty, \|g\|_\infty\leq t <1$, and Proposition \ref{prop:inner-approximation} can be applied again.
\end{itemize}
These two operations set up a ``ping-pong'' iteration, which produces a sequence of pairs of inner functions whose joint distributions approach the normalized Lebesgue measure on $\T^2$. The independent inner functions will be obtained by taking the limit. The details are presented in the following.

\begin{proof}[Proof of Theorem \ref{thm:main}: Independent inner functions]
Recall $M(\varepsilon)= \lfloor \varepsilon^{-1/2026}\rfloor$. Choose an $\varepsilon_1<e^{-2026}$ so that $CM(\varepsilon_1)^7\varepsilon_1^{3/2}< \varepsilon_1^{5/4}$; here $C$ is the universal constant appearing in Proposition \ref{prop:correction}. Applying Proposition~\ref{prop:inner-approximation} (to the case $f=g=0$),
we obtain inner functions $\varphi_1,\psi_1$, both vanishing at the
origin, such that
$$
    \|\mu_{\varphi_1,\psi_1}-1\|_{\cal W}<\varepsilon_1.
$$
For such a pair of inner functions, an application of Proposition \ref{prop:correction} gives functions $h,k \in zH^\infty$ that satisfy $$\|h\|_\infty, \|k\|_\infty \leq \sqrt{\varepsilon_1}, \quad \|h\|_2 + \|k\|_2 \leq CM(\varepsilon_1)^2\varepsilon_1, $$ and $$\sum_{1\leq m,n \leq M(\varepsilon_1)} \bigg| \ip{(\varphi_1+h)^m}{(\psi_1+k)^n}\bigg| \leq C M(\varepsilon_1)^7\varepsilon^{3/2}_1.$$
Let $t<1$ to be specified later, and put
\begin{align}\label{eq:rescale-functions}
    f_1:=\frac{t}{1+\sqrt{\varepsilon_1}}(\varphi_1+h),\qquad g_1:=\frac{t}{1+\sqrt{\varepsilon_1}}(\psi_1+k).
\end{align}
Then $f_1(0)=g_1(0)=0$ and $\|f_1\|_\infty,\|g_1\|_\infty \leq t <1.$
By \eqref{eq:Fourier-coefficient-balayage} and \eqref{eq:finite-correction} we have
\begin{align}
    \|\rho_{f_1,g_1}-1\|_{\cal W}/2
    &=\sum_{m,n\geq1}|\ip{f_1^m}{g_1^n}| \nonumber \\
    &\leq 
    \sum_{1\leq m,n\leq M(\varepsilon_1)}
       \left|\ip{(\varphi_1+h)^m}{(\psi_1+k)^n}\right|
    +\sum_{\substack{m,n\geq1\\\max\{m,n\}>M(\varepsilon_1)}}
       t^{m+n}.\label{eq:superlinear-iteration}
\end{align}
The first term does not exceed $CM(\varepsilon_1)^7\varepsilon_1^{3/2} < \varepsilon_1^{5/4}$. For the second term, we apply the estimate 
\begin{align*} 
\sum_{\substack{m,n\geq1\\\max\{m,n\}>M(\varepsilon_1)}}
t^{m+n} \leq \bigg(\sum_{m> M(\varepsilon_1)} t^m\bigg)\bigg(\sum_{n\geq 1} t^n\bigg)+ \bigg(\sum_{m\geq 1} t^m\bigg)\bigg(\sum_{n >  M(\varepsilon_1)} t^n\bigg)\lesssim \frac{t^{M(\varepsilon_1)}}{(1-t)^2}.
\end{align*} 
Therefore, by choosing $t= 1-1/\sqrt{M(\varepsilon_1)}$, and then using $M(\varepsilon)\sim \varepsilon^{-1/2026}$ and $xe^{-\sqrt{x}}\lesssim x^{-5000},$ one has $$\frac{t^{M(\varepsilon_1)}}{(1-t)^2} \lesssim M(\varepsilon_1) e^{-\sqrt{M(\varepsilon_1)}}\lesssim \varepsilon_1^{5/4}.$$ 
Substituting these estimates into \eqref{eq:superlinear-iteration}, we obtain
\begin{align}\label{eq:superlinear-Wiener}
\|\rho_{f_1,g_1}-1\|_{\cal W}\lesssim \varepsilon_1^{5/4}.
\end{align}
Meanwhile, a direct calculation using \eqref{eq:perturbation-bounds} gives
\begin{align*}
    \|f_1-\varphi_1\|_2
    &\leq \frac{1-t+\sqrt{\varepsilon_1}}{1+\sqrt{\varepsilon_1}}\|\varphi_1\|_2 + \frac{t}{1+\sqrt{\varepsilon_1}} \|h\|_2 \\
    & \lesssim M(\varepsilon_1)^{-1/2} + \varepsilon_1^{1/2} + M(\varepsilon_1)^2 \varepsilon_1 \\
    & \lesssim M(\varepsilon_1)^{-1/2} \asymp \varepsilon_1^{1/4052}.
\end{align*}
The same estimate holds for $\|g_1-\psi_1\|_2$ and hence 
\begin{align}\label{eq:superlinear-2-norm-iteration}
    1-\|f_1\|_2 \lesssim \varepsilon_1^{1/4052}, \quad 1-\|g_1\|_2
    \lesssim \varepsilon_1^{1/4052}.
\end{align}

Now we apply Proposition~\ref{prop:inner-approximation} to the pair $(f_1,g_1)$. By \eqref{eq:superlinear-Wiener}, there exists a pair of inner functions $(\varphi_2,\psi_2)$, both vanishing at the origin, such that 
\begin{align*} 
  \|\mu_{\varphi_2,\psi_2}-1\|_\cal W \leq  \|\mu_{\varphi_2,\psi_2}-\rho_{f_1,g_1}\|_\cal W + \|\rho_{f_1,g_1}-1\|_\cal W 
  < C \varepsilon_1^{5/4}.
\end{align*}
 Meanwhile, it follows from \eqref{eq:inner-approximation-norm-estimate} and \eqref{eq:superlinear-2-norm-iteration} that 
\begin{align*}
   \|\varphi_2 - \varphi_1\|_2  \leq \|\varphi_2-f_1\|_2 +\|f_1-\varphi_1\|_2 < C \varepsilon_1^{1/8104}, \quad \|\psi_2-\psi_1\|_2 < C \varepsilon_1^{1/8104}.
\end{align*}
Here $C$ is a fixed universal constant.

Now we set $\varepsilon_{j+1}:= C\varepsilon_j^{5/4}$. Following the above argument word for word, we can inductively
construct a sequence of pairs of inner functions
$(\varphi_j,\psi_j)_{j\geq1}$ such that
\begin{equation}\label{eq:iteration-estimate-8104}
\begin{gathered}
  \|\mu_{\varphi_j,\psi_j}-1\|_{\cal W}
  \leq \varepsilon_j,\\
  \|\varphi_{j+1}-\varphi_j\|_2
  \leq C\varepsilon_j^{1/8104},
  \qquad
  \|\psi_{j+1}-\psi_j\|_2
  \leq C\varepsilon_j^{1/8104}.
\end{gathered}
\end{equation}
Decreasing $\varepsilon_1$ if necessary, we may assume that $\varepsilon_{j+1}< \varepsilon_j/2$. In particular, both $(\varphi_j)_{j\geq 1}$ and $(\psi_j)_{j\geq 1}$ are Cauchy sequences in $zH^2$, and hence $$\varphi_j \longrightarrow \varphi, \quad \psi_j\longrightarrow \psi, \quad \text{in } zH^2.$$ We claim that $\varphi$ and $\psi$ are independent inner functions. Indeed, passing to an almost everywhere convergent subsequence, we see $\varphi$ and $\psi$ are inner. Moreover, applying the dominated convergence theorem along the same subsequence gives 
$$
    \int_{\T} \varphi^m\ol{\psi}^n \dd\sigma 
    =\lim_{j\to\infty} \int_\T \varphi_{j}^{m} \ol{\psi_j}^{n} \dd \sigma
    =0,
    \qquad m,n\geq1.
$$
This completes the proof of the first part of Theorem \ref{thm:main}.
\end{proof}

\subsection{The proof of Lemma \ref{lem:oscillation}}\label{sect:random-inner-function}

Recall that for a finite positive singular measure $\mu$ on $\T$, the boundary values of the associated singular inner function $S_\mu$ are given by
\begin{equation}\label{eq:singular-boundary}
	S_\mu(e^{it})=
	\exp\left(-i\,\operatorname{p.v.}\!\int_0^{2\pi}\cot\frac{t-x}{2}\,\dd\mu(e^{ix})
	\right)
	\quad\text{for a.e. }t\in[0,2\pi),
\end{equation}
where $\operatorname{p.v.}$ denotes the Cauchy principal value.
For each fixed $e^{it}\in\T$, we have
$$
\int_{\T\times(0,\infty)}
\left(1\wedge a\left| \cot\frac{t-x}{2}\right| 
\right)
\,\dd\nu(e^{ix},a)=\frac{\lambda}{\beta(1-\beta)}\int_{\T}\left| \cot\frac{t-x}{2}\right|  ^\beta\,\dd\sigma(e^{ix})
<\infty,
$$
and hence by \eqref{eq:poisson-integrability}, 
$\int_{\T}\left| \cot\frac{t-x}{2}\right| \,\dd\widetilde{\Pi}(e^{ix})<\infty$ almost surely. Then Fubini's theorem and \eqref{eq:singular-boundary} give that almost surely,
\begin{equation*}
	S_{\widetilde{\Pi}}(e^{it})=e^{-i\Phi(t)},
	\qquad \text{where \,}
	\Phi(t):=
	\int_0^{2\pi}\cot\frac{t-x}{2}\,\dd\widetilde{\Pi}(e^{ix}),
\end{equation*}
for a.e. $e^{it}\in\T$, with the integral converging absolutely. Therefore, we have
for a.e. $t,s\in[0,2\pi)$
$$
\E\left[ \cramped{S_{\widetilde{\Pi}}(e^{it})^m} \overline{S_{\widetilde{\Pi}}(e^{is})^m}\right]
=\E\exp\left(-im\bigl(\Phi(t)-\Phi(s)\bigr)\right).
$$
The exponential formula for Poisson random measures \cite{LastPenrose} states that if $h$ is a real-valued function satisfying
$\int (1\wedge |h|)\,\dd\nu<\infty$, then
\begin{equation}\label{eq:poisson-index}
	\E\exp\left(i\int h\,\dd\Pi\right)
	=\exp\left(\int (e^{ih}-1)\,\dd\nu\right).
\end{equation}
Taking $h(e^{ix},a)=-ma\bigl(\cot\frac{t-x}{2}-\cot\frac{s-x}{2}\bigr)$ in \eqref{eq:poisson-index}, one has
\begin{align*}
	& \,\left|
	\E\left[\cramped{S_{\widetilde{\Pi}}(e^{it})^m} 
	\overline{S_{\widetilde{\Pi}}(e^{is})^m}\right]
	\right| \\
	=&\,\exp\left(
	-\lambda\int_0^{2\pi}\int_0^\infty
	\left[1-\cos\left(
	ma\Big(\cot\frac{t-x}{2}-\cot\frac{s-x}{2}\Big)
	\right)\right]
	\frac{\dd a}{a^{1+\beta}}\,\dd\sigma(e^{ix})
	\right)\\
	=&\,\exp\left(-c_\beta\lambda m^\beta J_\beta(t,s)\right),
\end{align*}
where
$$
J_\beta(t,s):=\int_\T
\left|\cot\frac{t-x}{2}-\cot\frac{s-x}{2}\right|^\beta
\,\dd\sigma(e^{ix}),
\qquad
c_\beta:=\int_0^\infty(1-\cos a)\frac{\dd a}{a^{1+\beta}}.
$$
Put $\delta(t,s):=\min\{|t-s|,2\pi-|t-s|\}\in[0,\pi]$. For distinct \(t,s\in[0,2\pi)\), a direct computation gives
$$
J_\beta(t,s)=J_\beta(\delta,0)
\geq \frac{1}{2\pi}\int_{\delta/4}^{\delta/2}
\left(\cot\frac{\delta-x}{2}+\cot\frac{x}{2}\right)^\beta
\,\dd x
\geq \frac{1}{8\pi}\delta^{1-\beta}.
$$
This implies
\begin{equation}\label{eq:exp-Poisson-order}
	\left|
	\E\left[\cramped{S_{\widetilde{\Pi}}(e^{it})^m}
	\overline{S_{\widetilde{\Pi}}(e^{is})^m}\right]
	\right|
	\leq \exp\left(-\frac{c_\beta}{8\pi}\lambda m^\beta \delta(t,s)^{1-\beta}\right).
\end{equation}
Since $S_1$ and $S_2$ are independent, the left-hand side of \eqref{eq:oscillation} is
$$
\E\left[\left|\int_{\T}b\, \cramped{S_1^m} \overline{S_2^n}\dd\sigma\right|^2\right]
=\int_{\T}\int_{\T}b(z)\overline{b(w)}\,\E\left[\cramped{S_1(z)^m}\overline{S_1(w)^m}\right]\E\left[\overline{S_2(z)^n} \cramped{S_2(w)^n}\right]\dd\sigma(z)\dd\sigma(w).
$$
Substituting \eqref{eq:exp-Poisson-order}, we conclude that
$$
\begin{aligned}
	\E\left[\left|\int_{\T}
	b\,\cramped{S_1^m}\overline{S_2^n}\,\dd\sigma\right|^2\right]
	&\leq \|b\|_\infty^2
	\int_{\T}\int_{\T}
	\exp\left(-\frac{c_\beta}{8\pi}\lambda(m^\beta+n^\beta)\delta(t,s)^{1-\beta}\right)
	\,\dd\sigma(e^{it})\,\dd\sigma(e^{is})\\
	&\leq \frac{\|b\|_\infty^2}{\pi}
	\int_0^\pi
	\exp\left(-\frac{c_\beta}{8\pi}\lambda(m^\beta+n^\beta)u^{1-\beta}\right)
	\,\dd u\\
	&\leq \frac{\|b\|_\infty^2}{\pi}
	\int_0^\infty
	\exp\left(-\frac{c_\beta}{8\pi}\lambda(m^\beta+n^\beta)u^{1-\beta}\right)
	\,\dd u\\
	&=C\|b\|_\infty^2
	\bigl[\lambda(m^\beta+n^\beta)\bigr]^{-\frac{1}{1-\beta}},
\end{aligned}
$$
where $C$ is a constant depending only on $\beta$. This completes the proof.

\section{Independent Blaschke products}
\label{sec:independent-blaschke}

In this section, we explain how to refine the arguments in Section \ref{sec:independent-inner-functions} to obtain
independent Blaschke products, both vanishing at the origin. It should be noted that a naive application of Frostman's theorem does not suffice, since the resulting Blaschke products need not
vanish at the origin. To ensure that both inner functions at each stage of the iteration can be chosen to be Blaschke products, we refine Propositions~\ref{prop:inner-approximation} and~\ref{prop:correction} as follows.

\begin{proposition}
\label{prop:blaschke-approximation}
Let $B_1,B_2$ be Blaschke products satisfying $B_1(0)=B_2(0)=0.$
Suppose that
\[
  f\in B_1H^\infty,\qquad g\in B_2H^\infty,
  \qquad
  \|f\|_\infty<1,\qquad \|g\|_\infty<1.
\]
Then for every $\varepsilon>0$, there exist two Blaschke products $\varphi \in B_1H^\infty$ and $\psi\in B_2H^\infty$ such that
\begin{equation*}
  \mu_{\varphi,\psi}\in\cal W,
  \qquad
  \|\mu_{\varphi,\psi}-\rho_{f,g}\|_{\cal W}<\varepsilon.
\end{equation*}
Moreover, they can be chosen so that
\begin{equation*}
  \|\varphi-f\|_2^2<10(1-\|f\|_2^2),
  \qquad
  \|\psi-g\|_2^2<10(1-\|g\|_2^2).
\end{equation*}
\end{proposition}

\begin{proof}
Write $f=(B_1/z)\widetilde f$ and $g=(B_2/z)\widetilde g$. Then $\widetilde{f}$ and $\widetilde{g}$ are $H^\infty$-functions satisfying \eqref{eq:schur-function}, and hence the argument in the proof of Proposition \ref{prop:inner-approximation} can be adopted for $\widetilde{f}$ and $\widetilde{g}$. Similarly, we also first consider the case where $\widetilde{f}$ and $\widetilde{g}$ are polynomials of degree less than $N$. Then the only modification is to replace the random inner functions
in \eqref{eq:random-inner-approximation} by
$$
  \widetilde{\varphi}(z):=\tau_{\tilde{f}}(z,e^{i\theta_1}S_1(z)), \quad \widetilde{\psi}(z):= \tau_{\tilde{g}}(z,e^{i\theta_2}S_2(z))
$$
where $\theta_1,\theta_2$ are fresh independent uniform random variables on $[0,2\pi)$. The key point is that when $\widetilde{f}$ and $\widetilde{g}$ are nonzero polynomials of degree less than $N$, $\widetilde{\varphi}$ and $\widetilde{\psi}$ are almost surely Blaschke products. This fact follows from a standard argument using the classical characterization (see \cite[Section 5.3]{Rudin1969}): an inner function $I$ is a Blaschke product
if and only if
$$
  \lim_{r\to 1^-}
  \int_\T\log|I(rw)|\,d\sigma(w)=0.
$$
Following the same lines as the proof of Proposition~\ref{prop:inner-approximation}, and then setting $$
  \varphi:=\frac{B_1}{z}\widetilde\varphi,\qquad
  \psi:=\frac{B_2}{z}\widetilde\psi,
$$
one obtains the desired Blaschke products. The general case follows by approximation with nonzero polynomials, exactly as in the final part of the proof of Proposition~\ref{prop:inner-approximation}. We omit the details.
\end{proof}

\begin{proposition}
\label{prop:blaschke-correction}
Let $0<\varepsilon<e^{-2026}$ and $M(\varepsilon):=\lfloor\varepsilon^{-1/2026}\rfloor.$
Let $B_1,B_2$ be Blaschke products satisfying
$B_1(0)=B_2(0)=0$. Then for any given inner functions $\varphi$, $\psi$ satisfying $$
  \varphi\in B_1H^\infty,\qquad
  \psi\in B_2H^\infty,
  \qquad
  \|\mu_{\varphi,\psi}-1\|_{\cal W}<\varepsilon,$$
there exist small perturbations $h\in B_1H^\infty$, $k\in B_2H^\infty$ such that
\begin{equation*}
  \|h\|_2+\|k\|_2\leq CM^2\varepsilon,
  \qquad
  \|h\|_\infty,\ \|k\|_\infty\leq\sqrt{\varepsilon},
\end{equation*}
and
\begin{equation*}
  \sum_{1\leq m,n\leq M}
  \left|
    \left\langle
      (\varphi+h)^m,(\psi+k)^n
    \right\rangle
  \right|
  \leq CM^7\varepsilon^{3/2}.
\end{equation*}
Here the constant $C$ is universal.
\end{proposition}

The proof of Proposition~\ref{prop:blaschke-correction} requires no essential changes; one only needs to replace the domain of $T$ by
$B_1H^2\oplus\overline{B_2H^2}$ and repeat the argument. The final construction also requires a minor refinement.

\begin{proof}[Proof of Theorem~\ref{thm:main}: Independent Blaschke products]
Choose a small $\varepsilon_1<e^{-2026}$. Applying Proposition~\ref{prop:blaschke-approximation} (to the case 
$B_1=B_2=z$ and $f=g=0$) gives Blaschke products $\varphi_1,\psi_1$, both vanishing at the origin, such that $
  \|\mu_{\varphi_1,\psi_1}-1\|_{\cal W}<\varepsilon_1.$

If the pair $(\varphi_j,\psi_j)$ is determined, then one can apply
Proposition~\ref{prop:blaschke-correction} (to the case $B_1=\varphi_j$ and $B_2=\psi_j$) to obtain small perturbations $h_j\in\varphi_jH^\infty$ and $k_j\in\psi_jH^\infty$. Following the argument from \eqref{eq:rescale-functions} to \eqref{eq:iteration-estimate-8104}, one obtains Blaschke products $\varphi_{j+1},\psi_{j+1}$ satisfying 
\begin{equation*}
\begin{gathered}
  \|\mu_{\varphi_j,\psi_j}-1\|_{\cal W}
  \leq \varepsilon_j,\\
  \|\varphi_{j+1}-\varphi_j\|_2
  \leq C\varepsilon_j^{1/8104},
  \qquad
  \|\psi_{j+1}-\psi_j\|_2
  \leq C\varepsilon_j^{1/8104},
\end{gathered}
\end{equation*} where $\varepsilon_{j+1}=C\varepsilon_j^{5/4}$. As in the original proof, these limits are independent
inner functions vanishing at the origin. 

However, observe that the sequence of Blaschke pairs obtained
from Propositions~\ref{prop:blaschke-approximation}
and~\ref{prop:blaschke-correction} satisfies $$
  \varphi_j\mid\varphi_{j+1},
  \qquad
  \psi_j\mid\psi_{j+1},
  \qquad j\ge1.$$
These divisibility relations suffice to ensure that the limits are still Blaschke products. Indeed, divisibility and $H^2$ convergence imply that $\varphi_j\mid \varphi$ and $
\|\varphi/\varphi_j-1\|_2
  =\|\varphi-\varphi_j\|_2\to 0$.
Since each $\varphi_j$ is a Blaschke product, any singular
inner factor $S$ of $\varphi$ also divides $\varphi/\varphi_j$. As a consequence, $
  \left|\left(\varphi/\varphi_j\right)(0)\right|
  \le |S(0)|\le1.$
The left-hand side tends to $1$, so $S$ is constant and hence $\varphi$ is a Blaschke product. The same argument applies to $\psi$, completing the proof.
\end{proof}

\smallskip
\noindent\textbf{Acknowledgment.} Kunyu Guo is supported by the National Key R\&D Program of China (2024\\YFA1013400) and NSFC (No.~12231005). Dilong Li is supported by NSFC (No.~124B2005).

\bigskip
\begingroup
\small
\noindent\textbf{Xiangdi Fu}\\
School of Fundamental Physics and Mathematical Sciences,
HIAS, University of Chinese Academy of Sciences, Hangzhou 310024, China\\
\textit{Email address:}
\href{mailto:xdfu@ucas.ac.cn}{\nolinkurl{xdfu@ucas.ac.cn}}

\medskip
\noindent\textbf{Kunyu Guo}\\
School of Mathematical Sciences,
Fudan University, Shanghai 200433, China\\
\textit{Email address:}
\href{mailto:kyguo@fudan.edu.cn}{\nolinkurl{kyguo@fudan.edu.cn}}

\medskip
\noindent\textbf{Dilong Li}\\
School of Mathematical Sciences,
Fudan University, Shanghai 200433, China\\
\textit{Email address:}
\href{mailto:dlli23@m.fudan.edu.cn}{\nolinkurl{dlli23@m.fudan.edu.cn}}

\medskip
\noindent\textbf{Yanqi Qiu}\\
School of Fundamental Physics and Mathematical Sciences,
HIAS, University of Chinese Academy of Sciences, Hangzhou 310024, China\\
\textit{Email address:}
\href{mailto:yanqiqiu@ucas.ac.cn}{\nolinkurl{yanqiqiu@ucas.ac.cn}}
\endgroup

\end{document}